\documentclass[14pt]{article}
\usepackage[english]{babel}
\usepackage{amsfonts,amssymb,amsmath,amstext,amsbsy,amsopn,amscd,amsthm,graphicx,euscript,graphicx}
\usepackage{graphics}
\newtheorem{theorem}{Theorem}

\newtheorem{remark}{Remark}

\newcounter{tdfn}
\newenvironment{dfn}
{\vspace{0.15cm}{\bf Definition \arabic{tdfn}.}} {\par
\addtocounter{tdfn}{1}}
\newcounter{trk}
{\endtrivlist}

\def\:{\colon}

\def\R{{\mathbb R}}
\def\Z{{\mathbb Z}}
\def\0{{\mathbf 0}}
\def\1{{\mathbf 1}}

\def\R{{\mathbb R}}

\author{V.O.Manturov \footnote{Moscow Institute of Physics and Technology}, I.M.Nikonov\footnote{Moscow State University}}

\title{Maps from knots in the cylinder to flat-virtual knots}

\begin{document}

\maketitle

\section{Introduction}

Virtual knot theory invented by Kauffman \cite{VirtualKnots} in late nineties, has experienced
a lot of developments over the last two decades, among the most important of them, we can mention {\em parity}
and {\em picture-valued invariants}, \cite{Parity}.

The latter allow one to realise the principle saying that {\em if a diagram is complicated enough then it realises
itself} in a sense very close to that one have when working with free groups: the unique reduced word representing
a group element appears in any other element of this group.

In (virtual, free) knot theory, this principle is realised by the {\em parity bracket}, \cite{Parity}, an invariant
of knots valued in linear combinations of knot diagrams. For some (odd irreducible) knot diagrams one has the formula

$$[K]=K$$

meaning that the value of the invariant called bracket $[\cdot ]$ on a concrete knot diagram $K$ equals this diagram
with coefficient $1$. Taking any other diagram $K'$ equivalent to $K$, we get $[K']=[K]$, which, in turn, mean that
the diagram $K'$ ``contains'' $K$ as a subdiagram. Such an approach allows to judge about knots by looking at one particular
diagram and estimate various complexities which open a deeper insight than just numerical or polynomial invariants:
we can say much more than just a bare estimate of crossing number or a certain genus.

In a sense, the parity bracket is a variation of the Kauffman bracket where we smooth only {\em even} crossings
and evaluate the remaining diagrams (having former odd crossings) to themselves. In the first version of the
parity bracket, we are left with $\pm$ coefficients, however, many variations of it allow us to take coefficients
in the ring $\Z [a,a^{-1}]$ as the classical Kauffman bracket, see
\cite{IMN}.

Besides that, virtual knots have many invariants valued in free groups and free products of cyclic groups having
similar properties.

All such properties become possible because they use some intrinsic ``parity'' or ``non-trivial homology'' of the ambient
space where virtual knots live. Classical knots have no parity and no parity bracket.

However, classical knot theory has many instances of non-trivial homology and free groups, in particular, the braid
group configuration spaces and braid group faithful action on free groups.

In 2015, the author introduced a two-parameter family of groups called $G_{n}^{k}$ and formulated the following principle \cite{MN}.

If dynamical system describing a motion of $n$ particles possess a {\em nice codimension $1$ property governed
exactly by $k$ particles then such dynamics possess invariants valued in groups $G_{n}^{k}$}.
Without going into details of the groups $G_{n}^{k}$, we mention just one property for dynamics of motion
of $n$ points on the plane: for $k=3$ we can take the property ``three points are collinear''.

In some sense, the invention of the groups $G_{n}^{3}$ were an attempt to invent a substitute for
parity in the classical setup: the ambient space $\R^{2}$ (or $\R^{3}$) does not possess any homology,
but when dealing with {\em braids}, the punctured plane does contain homology groups.

Lots of invariants of braids appeared since that time; moreover, the $G_{n}^{k}$ theory works well for studying fundamental
groups of other configuration spaces.

The main obstruction from extending this theory from braids to knots was a very important condition that the number of
particles of dynamical system should stay fixed during the motion; in particular, a knot in $\R^{3}$ with its maxima
and minima does not satisfy such conditions.

In \cite{ManturovNikonovBraids}, we take the above approaches together and convert classical braids into braids in the cylinder,
and then to braid diagrams drawn in surfaces of higher genera.


In the present paper, we address the problem how to get a map from knots in the cylinder and on the thickened torus to some (generalisation of)
virtual knots called {\em virtual-flat knots.}

The main feature of this construction is that starting from objects with rather modest homology group ($H_{1}(S^{1}\times I^{1})=\Z$)
we can construct virtual knots of arbitrary high genus hence having lots of features.

The main construction takes a diagram on a cylinder (torus) and adds some ``invisible'' crossings which gives rise to a diagram
which can be formally immersed but not embedded (drawn) on the cylinder (torus) and living comfortably in thickened surfaces of higher
genera.

This allows one to ``pull back'' invariants of virtual theory to the theory of knots in the thickened cylinder (torus)
where the parity bracket and other picture-valued invariants are not strong enough.

This project was initiated in the paper \cite{ManturovNikonovBraids}

at the level of braids where a map from
classical braids to (a generalisation of) classical braids led to new representations of the classical braid group.

A {\em flat-virtual link diagram} is a four-valent graph on the plane
where each vertex is of one of the following three types:
\begin{enumerate}
\item classical (in this case one pair of edges is marked as an {\em overcrossing strand}).
\item flat;
\item virtual.
\end{enumerate}

The number of link components of a flat-virtual link diagram is just the number of unicursal
component of such. A {\em flat-virtual knot diagram} is a one-component flat-virtual link diagram.

The number of components is obviously invariant under the moves listed below, so, it will be reasonable
to talk about the number of components of a {\em flat-virtual link.}

\begin{dfn}
Here a {\em flat-virtual link} is an equivalence class of flat-virtual link diagrams modulo the following moves:
\begin{enumerate}
\item Classical Reidemeister moves which affect only classical crossings [there are three such moves in the non-orientable
case];
\item Flat Reidemeister moves which refer to flat crossings only and are obtained by forgetting the information about
overpasses and underpasses [the second Reidemeister move and the third Reidemeister move];
\item Mixed flat-classical Reidemeister-3 move when a strand containing two consecutive flat crossings
passes through a classical crossing of a certain type [the over-under structure at the unique classical crossing is preserved];
\item Virtual detour moves: a strand containing only virtual crossings and self-crossings can be removed
and replaced with a strand with the same endpoints where all new intersections are marked by virtual crossings.
\end{enumerate}

\end{dfn}

\begin{remark}
The last move can be expressed by a sequence of local moves containing purely virtual Reidemeister
moves and third mixed Reidemeister moves where an arc containing virtual crossings only goes through
a classical (or a flat) crossing.
\end{remark}

\begin{dfn}
By a {\em restricted flat virtual link} we mean an analogous equivalence class where
we forbid the third Reidemeister move with three flat virtual crossings.

\end{dfn}

\begin{remark}
Following Bar-Natan's point of view, virtual knots (links) can be considered as having only classical crossings
meaning that the Gauss diagram (multi-circle Gauss diagram) indicates the position of crossings on the plane
and the way how they have to be connected to each other (the way the connecting strands virtually intersect each other
does not matter because of the detour move).

The same way, one can define flat-virtual knot and link diagrams can be encoded by Gauss diagrams with
the only difference being that when a classical crossing carries two bits of information (over/undercrossing and
writhe) while the flat crossing only carries one bit of information (clockwise rotation).
\end{remark}

\begin{remark}
There is an obvious forgetful map from flat-virtual links to flat links: we just make all classical crossing
flat and forget the over-under information.
\end{remark}

Given a cylinder $C=S^{1}\times [0,1]$ with angular coordinate $\alpha\in [0,2\pi=0)$ and vertical coordinate $z\in [0,1]$.
Fix a number $d$ and say that two points $x,y$ on $c$ are {\em equivalent} if they have the same vertical coordinate and
their angular coordinate differs by $\frac{2\pi k}{d}$ for some integer $k$.
Likewise, for the torus $T^{2}$ with two angular coordinates $(\alpha,\beta)\in [0,2\pi=0)$ we fix a lattice $l$
as a discrete subgroup. Say, choosing $d_{1},d_{2}$ our group contains points with angular coordinates
$(\frac{2\pi k_{1}}{d_{1}},\frac{2\pi k_{2}}{d_{2}})$. We say that two points $A,B$ on the torus are {\em equivalent} if their
difference $A-B$ belongs to the subgroup $l$.

Now, having a diagram $K$ on the cylinder (torus) with a fixed number $d$ (lattice $l$) we define
a {\em flat-virtual} diagram $\phi_{d}(K)$ (resp., $\phi_{l}(K)$
(up to virtual detour moves) as follows. First, we assume without loss of generality that
$K$ is {\em generic with respect to the subgroup}, i.e., for any two distinct equivalent points $e,e'$ belonging
to edges are not crossing points and the tangent vectors at these edges are transverse.

In both cases (for $\phi$ and for $\phi'$) for each pair of equivalent points on the edges
(say, $(e_{j},e'_{j})$ we create a new {\em flat} crossing in the following manner.
Let $e'_{j}= e_{j}+g$, where $g$ is the element of the corresponding group ($\Z_{d}$ or $l$).
We choose one
of these points (never mind which one, say, $e_{j}$), remove a small neighbourhood $N=U(e_{j})$
with endpoints $A,B$, shift it by $g$ and connect $A$ to $A+g$ and $B$ to $B+g$ by any curves on the
plane not passing through any crossings (all crossings on the newborn curves $[A,A+g], [B,B+g]$
are declared to be virtual (Fig.~\ref{fig:main_map}).

\begin{figure}
\centering\includegraphics[width=0.4\textwidth]{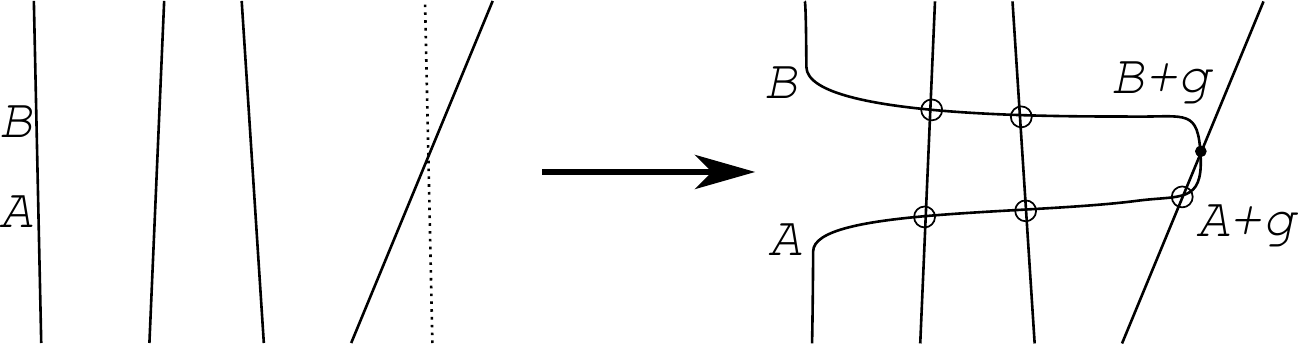}
\caption{The map $\phi$}\label{fig:main_map}
\end{figure}

\begin{theorem}
The maps $\phi_{d},\phi_{l}'$ defined above are well defined maps from links on the thickened cylinder (resp., thickened torus)
to flat-virtual links. In other words, if $K$ and $K'$ are isotopic diagrams on the cylinder (torus)
then $\phi_{d}(K)\sim \phi_{d}(K')$
(resp., $\phi'_{l}(K)\sim \phi'_{l}(K)$) where
$\sim$
denotes the equivalence of flat-virtual links.
\end{theorem}

Moreover, for the special case of knots in the full torus for $d=2$ we have a more exact statement
\begin{theorem}
If $K$ and $K'$ are isotopic diagrams on the cylinder then
$\phi(K)$ and $\phi(K)$ are diagrams of the same restricted flats virtual link.
\end{theorem}

\section{Examples}

Below, we give a series of examples.

Consider the Whitehead link shown in Fig.\ref{whitehead}, upper-left part. Taking away its trivial (red) component, we get the following link in the full torus, see top-right picture.

\begin{figure}
\centering\includegraphics[width=0.6\textwidth]{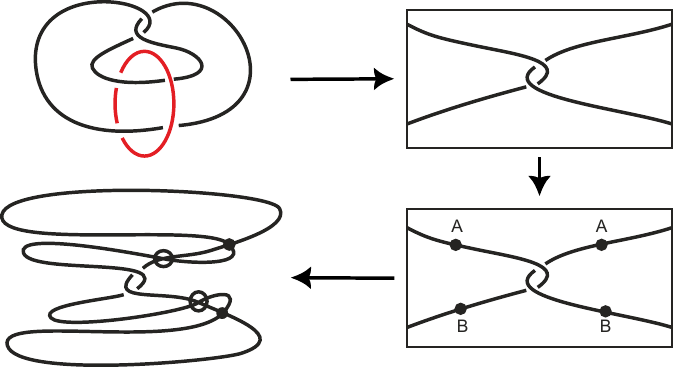}
\caption{The knot on the full torus}
\label{whitehead}
\end{figure}

If we take $d=2$, we shall see two pairs of points ($A,A$ and $B,B$) which have to be pasted together,
down-right. After creating corresponding flat crossings (and forgetting the full torus structure),
we get a flat-virtual diagram in the bottom-left picture.

Let us call the obtained flat-virtual knot diagram $W$. It has two flat crossings. It is not hard to see that the corresponding flat-virtual knot is non-trivial. The main reason is that the first Reidemeister move is not allowed for flat virtual knots.



In order to show that the flat-virtual knot $W$ is not trivial, we consider a polynomial invariant of flat-virtual links --- {\em flat-virtual Jones polynomial}.

Let $D$ be a flat-virtual oriented link diagram. Denote the set of classical crossings of $D$ by $C(D)$. Consider the set of {\em states} $S=\{0,1\}^{C(D)}$. Any state $s\in S$ determines a smoothing $D_s$ of the diagram $D$ which is obtained from $D$ by the smoothing rule given in Fig.~\ref{fig:smoothing_rule}.

\begin{figure}
\centering\includegraphics[width=0.5\textwidth]{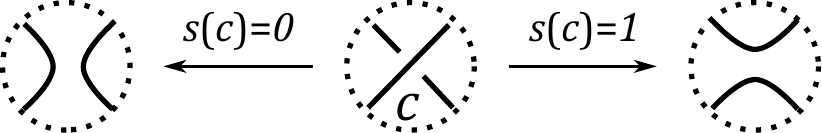}
\caption{Smoothing rule at a crossing $c$}\label{fig:smoothing_rule}
\end{figure}

For a given state $s$, denote
\[
\alpha(s)=\#\{c\in C(D)\mid s(c)=0\},\quad \beta(s)=\#\{c\in C(D)\mid s(c)=1\}.
\]
Let $\gamma_{even}(s)$ (resp. $\gamma_{odd}(s)$) be the number of components in $D_s$ which have even (resp. odd) number of flat crossings (note that we count a flat self-crossing of a component only once). Let $w(D)=\sum_{c\in C(D)}w(c)$ be the writhe of the diagram $D$.

Define the {\em flat-virtual Jones polynomial} $X(D)\in\Z[a,a^{-1},b]$ of the diagram $D$ by the formula
\begin{equation}\label{eq:flat_Jones}
X(D)=(-a)^{-3w(D)}\sum_{s\in S}a^{\alpha(s)-\beta(s)}(-a^2-a^{-2})^{\gamma_{even}(s)}b^{\gamma_{odd}(s)}.
\end{equation}

\begin{theorem}
The flat-virtual Jones polynomial $X$ is an invariant of flat-virtual links.
\end{theorem}

\begin{proof}
The proof goes along the lines of the proof of invariance for the classical Kauffman bracket and Jones polynomial~\cite{VirtualKnots}.

Let $D$ and $D'$ be two oriented flat-virtual link diagrams which differ by a Reidemeister move. If the move is classical one or a detour move then it does not affect the number of odd components in the smoothed diagrams $D_s$ and $D'_s$. Hence, we can use classical reasonings here.

If the move is a flat second Reidemeister move then it changes by $2$ the number of flat crossings on one or two components in the smoothed diagrams $D_s$, so the parity of the component remain unchanged. Then the expressions for $X(D)$ and $X(D')$ coincide.

A flat third Reidemeister move does not change the expression for $X(D)$.

Finally, a mixed flat third Reidemeister move may change by $2$ the number of flat crossings on two components in smoothed diagrams (Fig.~\ref{fig:smooth_reidmove3}). Hence, the parity of components in $D_s$ does not change, so does the value $X(D)$.
\begin{figure}
\centering\includegraphics[width=0.4\textwidth]{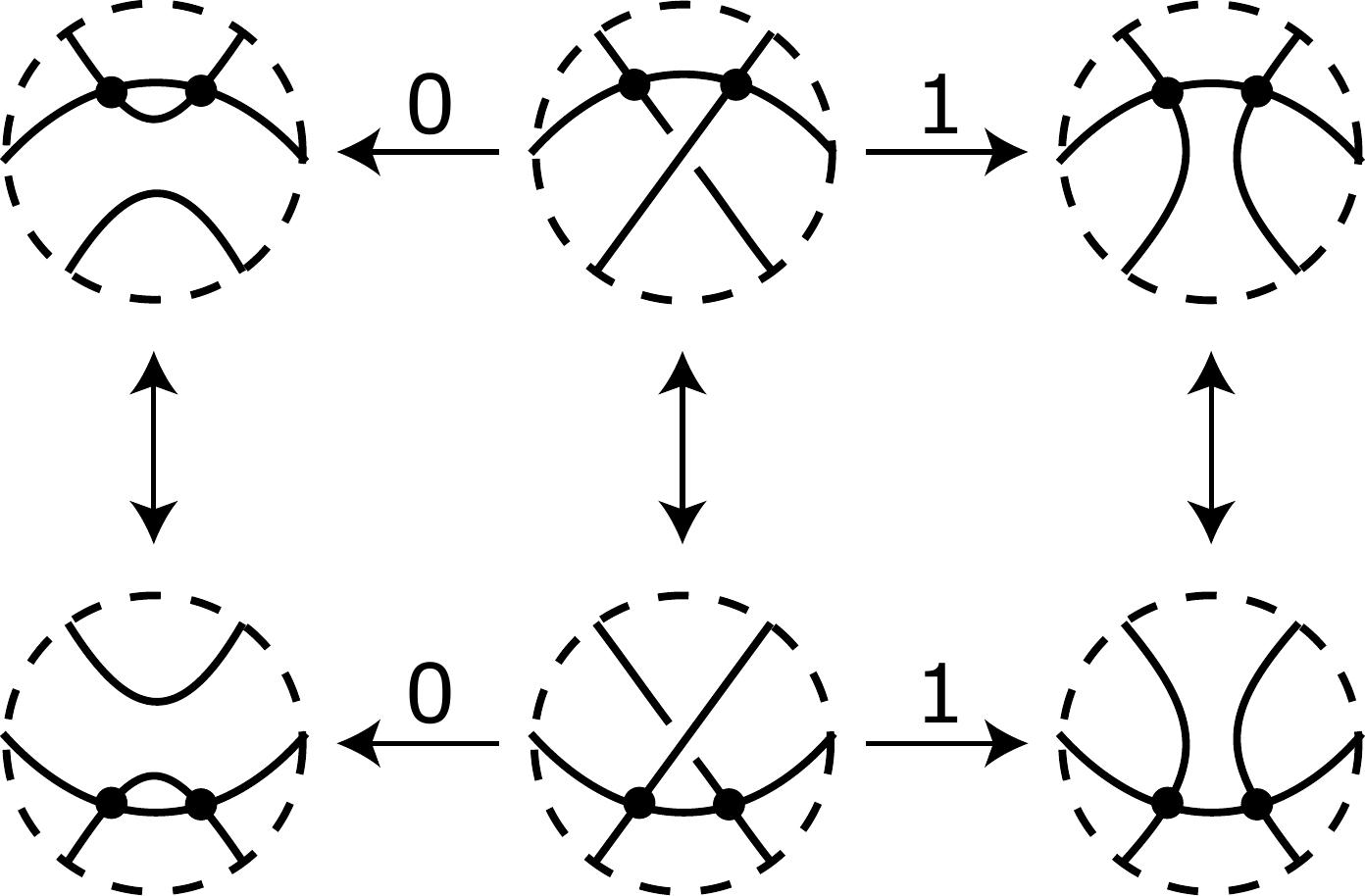}
\caption{Smoothings of a mixed third Reidemeister move}\label{fig:smooth_reidmove3}
\end{figure}
\end{proof}

For the flat-virtual knot $W$ we have
\begin{multline*}
X(W)=(-a)^{-6}[a^2(-a^2-a^{-2})b^2+2b^2+a^{-2}(-a^2-a^{-2})]=(a^{-6}-a^{-2})b^2-a^{-10}-a^{-6}.
\end{multline*}

The inequality $X(W)\ne 1$ shows that the Whitehead link is not split.

\end{document}